\documentclass{amsart}
\usepackage{amsfonts,amssymb,amsmath,enumerate,verbatim,mathtools,tikz,bm,mathrsfs,tikz-cd,hyperref}
\hypersetup{
    colorlinks=true,
    linkcolor=blue,
    filecolor=magenta,      
    urlcolor=cyan,
}
\usepackage[utf8]{inputenc}
\usepackage[english]{babel}
\usetikzlibrary{matrix}
\usetikzlibrary{arrows}

\usepackage[all]{xy}
\SelectTips{cm}{}
\usepackage{stmaryrd}

\DeclareMathOperator{\cone}{cone}

\newcommand{\inca}{\hookrightarrow}

\DeclareMathOperator{\h}{H}
\newcommand{\n}{\mathfrak{n}}

\newcommand{\VS}{\mathsf{VS}}

\newcommand{\VV}{\mathcal{V}}

\newcommand{\Z}{\mathbb{Z}}

\newcommand{\A}{\mathcal{A}}

\newcommand{\T}{\mathsf{T}}
\newcommand{\Poly}{\mathcal{S}}
\newcommand{\D}{\mathsf{D}}
\newcommand{\N}{\mathbb{N}}
\newcommand{\ov}[1]{\overline{#1}}
\newcommand{\vp}{\varphi}

\newcommand{\con}{\subseteq}

\pgfdeclarelayer{bg}    
\pgfsetlayers{bg,main} 
\newcommand{\x}{{\bm{x}}}
\newcommand{\g}{{\bm{g}}}
\newcommand{\del}{\partial}
\newcommand{\f}{{\bm{f}}}

\newcommand{\m}{\mathfrak{m}}
\newcommand{\p}{\mathfrak{p}}

\newcommand{\e}{\epsilon}

\DeclareMathOperator{\pd}{pd}

\DeclareMathOperator{\id}{id}

\DeclareMathOperator{\Spec}{Spec}
\DeclareMathOperator{\Proj}{Proj}
\DeclareMathOperator{\Supp}{Supp}

\DeclareMathOperator{\Hom}{Hom}
\DeclareMathOperator{\Ext}{Ext}

\DeclareMathOperator{\V}{V}
\DeclareMathOperator{\Kos}{Kos}

\newcommand{\gsupp}[1]{\mathsf{Supp}^+_{#1}}
\DeclareMathOperator{\Thick}{\mathsf{Thick}}

\newcommand{\shift}{{\mathsf{\Sigma}}}

\newcommand{\xra}{\xrightarrow}

\newtheorem{theorem}{Theorem}[subsection]
\newtheorem*{Theorem}{Theorem}

\newtheorem{proposition}[theorem]{Proposition}

\newtheorem{lemma}[theorem]{Lemma}

\theoremstyle{definition}

\newtheorem{definition}[theorem]{Definition}

\theoremstyle{remark}
\newtheorem{remark}[theorem]{Remark}

\newtheorem*{ack}{Acknowledgements}

\theoremstyle{theorem}
\newtheorem{stheorem}{Theorem}[section]

\newtheorem{sproposition}[stheorem]{Proposition}

\newtheorem{slemma}[stheorem]{Lemma}
\newtheorem{scorollary}[stheorem]{Corollary}
\newtheorem{squest}[stheorem]{Question}

\theoremstyle{definition}

\theoremstyle{remark}
\newtheorem{sremark}[stheorem]{Remark}

\newtheorem{schunk}[stheorem]{}
\newtheorem{chunk}[theorem]{}

\begin{document}

\title[The derived category of a locally complete intersection ring]{The derived category of a locally complete intersection ring}

\author[Joshua Pollitz]{Josh Pollitz}
\address{Department of Mathematics,
University of Nebraska, Lincoln, NE 68588, U.S.A.}
\email{jpollitz@huskers.unl.edu}

\date{\today}

\thanks{The author was partly supported through NSF grant DMS 1103176.}

\keywords{local ring, complete intersection, derived category, DG algebra, thick subcategory, support variety}
\subjclass[2010]{13D09, 13D07 (primary);  18G55, 18E30 (secondary)}

\begin{abstract}
In this paper, we answer a question of Dwyer, Greenlees, and Iyengar  by proving  a local ring $R$ is a complete intersection if and only if every complex of $R$-modules  with finitely generated homology is proxy small. Moreover, we establish that a commutative noetherian ring   $R$ is  locally a complete intersection  if and only if every complex of $R$-modules with finitely generated homology is virtually small. \end{abstract}

\maketitle

\section{Introduction}
The relation of the structure of a commutative noetherian ring $R$ and that of its category of modules has long been a major topic of study in commutative algebra. More recently, it has been extended to studying the relations between the structure of $R$ and that of its derived category $\D(R)$. 
 Working in this setting allows one to use ideas from algebraic topology and triangulated categories to gain insight into properties of  $R$. 

Basic information on $\D(R)$ is contained in its full subcategory consisting of complexes with finitely generated homology, denoted $\D^f(R)$. A  complex of $R$-modules is said to be \emph{perfect} if it is quasi-isomorphic to a  bounded complex of finitely generated projective $R$-modules.     The following homotopical characterization of regular rings is well known: \emph{a commutative noetherian ring $R$ is regular if and only if every object of $\D^f(R)$ is a perfect complex}. 

In many respects, the local rings that are closest to being regular are complete intersections.  We characterize of complete intersections  in terms of  how each object of $\D^f(R)$ relates to the perfect complexes. Moreover, this yields a homotopical characterization of a locally complete intersection ring.  Following \cite{DGI2} and \cite{DGI}, we say that a complex of $R$-modules $M$  \emph{finitely builds} a complex of $R$-modules $N$ provided that $N$ can be obtained by taking finitely many cones and retracts  starting from $M$.  More precisely, $M$  finitely builds $N$ provided $N$ is in $\Thick_{\D(R)}M$ (see  Section \ref{thick}).  The main result of the paper is the following:

\begin{Theorem}
A commutative noetherian ring   $R$ is  locally a complete intersection   if and only if every nontrivial object of $\D^f(R)$  finitely builds a nontrivial perfect complex. 
\end{Theorem}




\begin{ack}
I thank my advisors Luchezar Avramov and Mark Walker for their support and interesting conversations on this project.  I would also like to thank Srikanth Iyengar for several useful discussions regarding this work.
\end{ack}

\section{Preliminaries}


\subsection{Differential Graded Algebra}

 Fix a commutative noetherian ring $Q$. Let $A=\{A_i\}_{i\in \Z}$ denote a DG $Q$-algebra. We only consider  left DG $A$-modules. 

 \begin{chunk}Let $M$ and $N$ be  DG $A$-modules. We say that $\vp: M\to N$ is a \emph{morphism of DG $A$-modules} provided $\vp$ is a morphism of the underlying complexes of $Q$-modules such that   $\vp(am)=a\vp(m)$ for all $a\in A$ and $m\in M$.  We  write  $\vp: M\xra{\simeq} N$ when $\vp$ is a quasi-isomorphism.  
\end{chunk}

\begin{chunk} Let $M$ be a DG $A$-module. The differential of  $M$ is denoted by  $\del^M$. For each $i\in \Z$,  $\shift^i M$ is the DG $A$-module given by
 $$(\shift^i M)_{n}:=M_{n-i}, \  \del^{\shift^i M}:=(-1)^i\del^{M}, \ 
\text{ and } a\cdot m:=(-1)^{|a|i}am.$$  
Let $\h(M):=\{\h_i(M)\}_{i\in \Z}$ which is a graded module over the graded $Q$-algebra $\h(A):=\{\h_i(A)\}_{i\in \Z}.$  We let $M^\natural$ denote the underlying graded $Q$-module. Note that  $A^\natural$ is a graded $Q$-algebra and $M^\natural$ is a graded $A^\natural$-module.
\end{chunk}

\begin{chunk}
A DG $A$-module $P$ is  \emph{semiprojective} if for every morphism of DG $A$-modules $\alpha: P\to N$ and each surjective quasi-isomorphism of DG $A$-modules $\gamma: M\to N$ there exists a unique up to homotopy morphism of DG $A$-module $\beta: P\to M$ such that $\alpha=\gamma\beta$. \end{chunk}
\begin{chunk}\label{sres}
  A \emph{semiprojective resolution} of a DG $A$-module  $M$  is a surjective quasi-isomorphism of DG $A$-modules $\e: P\to M$ where  $P$ is a semiprojective DG $A$-module.   Semiprojective resolutions exist and any two semiprojective resolutions of $M$ are unique up to  homotopy equivalence \cite[6.6]{FHT}.  \end{chunk}

  \begin{chunk} \label{ext2}
  For DG $A$-modules $M$ and $N$, define $$\Ext_A^*(M,N):=\h(\Hom_A(P,N))$$ where $P$ is a semiprojective resolution of $M$ over $A$. Since any two semiprojective resolutions of $M$ are homotopy equivalent, $\Ext_A^*(M,N)$ is independent of choice of $M$. 
   An element $[\alpha]$ of $\Ext_A^*(M,N)$  is the class of a morphism of DG $A$-modules $$\alpha: P\to \shift^{|\alpha|} N.$$  Moreover, given $[\alpha]$ and $[\beta]$ in $\Ext_A^*(M,N)$, then $[\alpha]=[\beta]$ if and only if $\alpha$ and $\beta$ are homotopic morphisms of DG $A$-modules. 
\end{chunk}

\begin{chunk}
Let $\D({A})$ denote the derived category of ${A}$ (see \cite{Keller} for an explicit construction).  Recall that $\D(A)$, equipped with $\shift$, is a triangulated category. Define  $\D^f({A})$ to be the full subcategory of $\D({A})$ consisting of all $M$ such that $\h(M)$ is a finitely generated graded module over  $\h({A}).$ We use $\simeq$ to denote isomorphisms in $\D(A)$ and reserve $\cong$ for isomorphisms of DG $A$-modules. 
\end{chunk}

\subsection{Koszul Complexes}\label{Koszul}

Fix a commutative noetherian ring $Q$. Let $\f=f_1,\ldots, f_n$ be a list of elements in  $Q$. Define $\text{Kos}^Q(\f)$ to  be the DG $Q$-algebra with $\text{Kos}^Q(\f)^\natural$ the exterior algebra on a free $Q$-module with basis $\xi_1,\ldots,\xi_n$ of homological degree 1, and differential $\del\xi_i=f_i.$ We  write $$\text{Kos}^Q(\f)=Q\langle \xi_1,\ldots,\xi_n|\del\xi_i=f_i\rangle.$$

\begin{chunk}\label{kosaction}
Let $\f'=f_1',\ldots, f_m'$ be in $Q$. Assume  there exists $a_{ij}\in Q$ such that $$f_i=\sum_{j=1}^m a_{ij} f_j'.$$ There exists a unique morphism of  DG $Q$-algebras $\Kos^Q(\f)\to \Kos^Q(\f')$  satisfying $$\xi_i\mapsto \sum_{j=1}^m a_{ij} \xi_j'.$$ Therefore, $\Kos^Q(\f')$ is a DG $\Kos^Q(\f)$-module where the action is given by $$\xi_i\cdot e'=\sum_{j=1}^ma_{ij}\xi_j' e'$$ for all $e'\in E'.$ 
\end{chunk}

\begin{chunk}Assume that $(Q,\n,k)$ is a commutative noetherian local ring. Define $K^Q$ to be  the Koszul complex on some minimal generating set for $\n$. Then $K^Q$ is unique up to  DG $Q$-algebra isomorphism. 
\end{chunk}

\subsection{Map on  Ext}
Let $Q$ be a commutative noetherian ring. Fix a morphism of DG $Q$-algebras $\vp: A'\to A$.  Let $M$ and $N$ be DG $A$-modules,   $\e:P\to M$ be a semiprojective resolution of $M$ over $A$, and $\e':P'\to M$ a semiprojective resolution of $M$ over $A'$. There exists a unique up to homotopy morphism of DG $A'$-modules $\alpha: P'\to P$ such that $\e'=\e\alpha$.
Define $\Hom_\vp(\alpha,N)$ to be the composition $$\Hom_A(P,N)\xra{\Hom_\vp(P,N)} \Hom_{A'}(P,N)\xra{\Hom_{A'}(\alpha,N)}\Hom_{A'}(P',N).$$ This induces a map in cohomology  $$\Ext_\vp^*(M,N): \Ext_{A}^*(M,N)\to \Ext_{A'}^*(M,N)$$ given by $\Ext_\vp^*(M,N)=\h(\Hom_{\vp}(\alpha,N));$
it is independent of choice of $\alpha,$ $P$, and $P'$. 

\begin{chunk}\label{augiso}
Let $\vp: A'\to A$ be a morphism of DG $Q$-algebras and let $M$ and $N$ be DG $A$-modules. If $\vp$ is a quasi-isomorphism, then $\Ext_\vp^*(M,N)$ is an isomorphism \cite[6.10]{FHT}. 
\end{chunk}

In the following theorem,  the theory of  DG $\Gamma$-algebras is used. See \cite[Section 6]{IFR} or \cite[Chapter 1]{GL} as a reference for definitions and notation. 
\begin{theorem}\label{surj}Assume $(Q,\n,k)$ is a regular local ring. Let $R=Q/I$ where $I$ is minimally generated by  $\f=f_1,\ldots, f_n\in\n^2$. Let $E$ be the Koszul complex on $\f$ over $Q$. 
Let $\vp: E\to R$ denote the augmentation map. The canonical map $$\Ext_{\vp}^*(k,k):\Ext_R^*(k,k)\to \Ext_E^*(k,k)$$ is surjective. 
\end{theorem}
\begin{proof}
Write $E=Q\langle \xi_1,\ldots,\xi_n|\del \xi_i=f_i\rangle.$ For an element $a\in Q$,  let $\ov{a}$ denote the image of $a$ in $R$. Let  $s_1,\ldots, s_e$ be a minimal generating set for $\n$.  Let $X=\{x_1,\ldots, x_e\}$ be a set  of exterior variables of homological degree 1 and $Y=\{y_1,\ldots, y_n\}$  a set   of  divided power variables of homological degree 2.  By \cite[7.2.10]{IFR}, the morphism of DG $\Gamma$-algebras $\vp: E\to R$ extends to a morphism of DG $\Gamma$-algebras $$\vp\langle X\rangle: E\langle X|\del x_i=s_i\rangle\to R\langle X|\del x_i=\ov{s}_i\rangle$$ such that $\vp\langle X\rangle (x_i)=x_i$ for each $1\leq i \leq e$.  

Since $f_i\in \n^2$, there exists $a_{ij}\in \n$ such that $$f_i=\sum_{j=1}^e a_{ij} s_j.$$  For each $1\leq i \leq n$, we have degree 1 cycles $$z_i:=\sum_{j=1}^n a_{ij} x_j-\xi_i \ \text{ and } \ \ov{z}_i:=\sum_{j=1}^n \ov{a}_{ij} x_j$$  in $E\langle X\rangle$ and  $R\langle X\rangle,$ respectively, where $\vp\langle X\rangle(z_i)=\ov{z_i}.$  Applying \cite[7.2.10]{IFR}  yields a morphism of DG $\Gamma$-algebras $$\vp\langle X,Y\rangle: E\langle X\rangle\langle Y| \del y_i=z_i\rangle\to R\langle X\rangle\langle Y| \del y_i=\ov{z_i}\rangle$$ extending $\vp\langle X\rangle$ such that $\vp\langle X,Y\rangle (y_i)=y_i$ for each $1\leq i \leq n$. 

By \cite[6.3.2]{IFR}, $E\langle X,Y\rangle$ is an acyclic closure of $k$ over $E$. In particular, $E\langle X,Y\rangle$ is a semiprojective resolution of $k$ over $E$. Next,  $\ov{s}_1,\ldots, \ov{s}_e$ is a minimal generating set for the maximal ideal of $R$. Also, since $f_1,\ldots, f_n$ minimally generates $I$,  it follows that $[\ov{z}_1],\ldots, [\ov{z}_n]$ is  a minimal generating set for $H_1(R\langle X\rangle)$  (see \cite[Theorem 4]{Tate} or \cite[1.5.4]{GL}).  Thus, $R\langle X,Y\rangle$ is the second step in forming an acyclic closure of $k$ over $R$. Let $\iota: R\langle X,Y\rangle \inca R\langle X,Y,V\rangle$ denote the inclusion of DG $\Gamma$-algebras where $R\langle X,Y,V\rangle$ is an acyclic closure of $k$ over $R$ and $V$ consists of $\Gamma$-variables of homological degree at least 3. 
Define $\alpha: E\langle X,Y\rangle\to R\langle X,Y,V\rangle$ to be the morphism of DG $\Gamma$-algebras given  by $\alpha: =\iota\circ \vp\langle X,Y\rangle.$

The following  is a commutative diagram of $\Gamma$-algebras 
\begin{center}\begin{tikzcd}
E\langle X,Y\rangle\otimes_Ek \arrow{d}[swap]{\cong} \arrow{r}{\alpha\otimes k}   & R\langle X,Y,V\rangle\arrow{d}{\cong}\otimes_R k\\
k\langle X,Y\rangle  \arrow{r}{\con} &  k\langle X,Y,V\rangle
\end{tikzcd}\end{center}
 Therefore,  $\alpha\otimes k$  is an  injective  morphism   of  $\Gamma$-algebras. In particular, $\alpha\otimes k$ is injective as a map of graded $k$-vector spaces. 
Also, the following is a commutative diagram of graded $k$-vector spaces
 \begin{center}\begin{tikzcd}
\Hom_k(R\langle X,Y,V\rangle\otimes_Rk, k) \arrow{d}[swap]{\cong} \arrow{r}{(\alpha\otimes k)^*}  & \Hom_k(E\langle X,Y\rangle\otimes_E k,k)\arrow{d}{\cong}\\
\Hom_R(R\langle X,Y,V\rangle,k)  \arrow{r}{\Hom_\vp(\alpha,k)}  &  \Hom_E(E\langle X,Y\rangle,k)
\end{tikzcd}\end{center}  Since $\alpha\otimes k$ is injective,  $(\alpha\otimes k)^*$ is surjective. Thus,  $\Hom_\vp(\alpha,k)$ is surjective. Moreover,   $\Hom_E(E\langle X,Y\rangle,k)$ and $\Hom_R(R\langle X,Y,V\rangle,k)$  have trivial differential  (see \cite[6.3.4]{IFR}). Thus, $\Ext_{\vp}^*(k,k)=\Hom_\vp(\alpha,k)$, and so $\Ext_\vp^*(k,k)$ is surjective. 
\end{proof}


\subsection{Support of a Complex of Modules}

Let $R$ be a commutative noetherian ring and  $\Spec R$ denote the set of prime ideals of $R$. For a complex of $R$-modules $M$,  define the \emph{support of $M$} to be $$\Supp_RM:=\{\p\in \Spec R: M_\p \not\simeq 0\}.$$


\begin{chunk}\label{kossup}
Let $M$ be in $\D^f(R)$ and let  $\x$ generate an ideal $I$ of $R$. It follows from Nakayama's lemma that  $$\Supp_R (M\otimes_R \Kos^R(\x))=\Supp_R M\cap \Supp_R(R/I).$$  In particular, if $\x$ generates a maximal ideal $\m$ of $R$ with $\m\in \Supp_R M$, then $$\Supp_R (M\otimes_R \Kos^R(\x))= \{\m\}.$$ \end{chunk}

\begin{lemma}\label{ksp}
Let $n$  be a nonzero integer and let $M$ be in $\D^f(R)$. If  $\alpha:M\to \shift^n M$ is a morphism in $\D(R)$, then  $$\Supp_R M=\Supp_R(\cone(\alpha)).$$ \end{lemma}
\begin{proof}
Let $C:=\cone(\alpha)$. We have an  exact triangle $$ M\to \shift^{n}M\to  C\to $$ in $\D(R)$. For each $\p\in \Spec R$, there is an  exact triangle $$M_\p\to \shift^n M_\p\to C_\p\to $$ in $\D(R_\p)$. It follows that $\Supp_R C\con \Supp_R M$.

 If $\p\notin\Supp_R C,$ then $M_\p {\simeq} \shift^n M_\p$ in $\D(R_\p)$. Since $ M_\p {\simeq}\shift^n  M_\p$,  $M_\p$ is in  $\D^f(R_\p)$, and $n\neq 0$, it follows that $M_\p\simeq 0.$  Thus, $\p\notin \Supp_R M$. 
\end{proof}

\subsection{Thick Subcategories}\label{thick}

Let $\T$ denote a triangulated category.  A full subcategory $\T'$ of $\T$ is called  \emph{thick} if it is closed under suspension,   has the two out of three property on exact triangles, and is closed under direct summands. 
For an object $X$ of $\T$, define  the \emph{thick closure of $X$ in $\T$}, denoted $\Thick_\T X$, to be the intersection of all thick subcategories of $\T$ containing $X$. Since an intersection of thick subcategories is a thick subcategory, $\Thick_{\T}X$ is the smallest thick subcategory of $\T$ containing $X$. See \cite[Section 2]{HPC} for an inductive construction of $\Thick_{\T}X $ and a discussion of the related concept of \emph{levels}. If $Y$ is an object of $\Thick_\T X$, then we say that \emph{$X$ finitely builds $Y$}.

\begin{chunk}\label{thick2} Let $R$ be a commutative ring.  Recall that a complex of $R$-modules $M$ is \emph{perfect} if it is quasi-isomorphic to a bounded complex of finitely generated projective $R$-modules.  By \cite[3.7]{DGI}, $\Thick_{\D(R)}R$ consists exactly of the perfect complexes.   \end{chunk}
\begin{chunk}\label{thick3} Let $R$ be a commutative  ring and let $\m$ be a maximal ideal of $R$. By  \cite[3.10]{DGI}, $\Thick_{\D(R)}(R/\m)$ consists of all objects $M$ of $\D^f(R)$  such that $\Supp_R M=\{\m\}.$ 
\end{chunk}

\begin{chunk}\label{triangles2}
Let $F: \T\to \T'$ be  an exact  functor   between triangulated categories with right adjoint exact functor $G$.  Let $\e: FG\to \id_{T'}$ and $\eta:\id_T\to GF$ be the co-unit and unit transformations.

The full subcategory of $\T$ consisting of all objects  $X$ such that the natural map  $\eta_X: X\to GF(X)$ is an isomorphism is  a thick subcategory of $\T$. For each $X$ in $\T$, the composition $$F(X)\xra{F(\eta_X)}FGF(X)\xra{\e_{F(X)}} F(X)$$ is an isomorphism. Therefore,  if $\eta_X$ is an isomorphism in $\T$ then $ \e_{F(X)}$ is an isomorphism in $\T'$ and $F$ induces an equivalence of categories $$\Thick_\T X\xra{\cong} \Thick_{\T'}F(X).$$
\end{chunk}

\begin{lemma}\label{fl}Let $\vp: R\to S$ be flat morphism of   commutative  rings. Suppose $M$ is in $\D(R)$ and the natural map $M\to M\otimes_R S $ is an isomorphism in $\D(R)$. Then the  functor  $-\otimes_R S:\D(R)\to \D(S)$   induces an equivalence of categories $$\Thick_{\D(R)}M\xra{\cong } \Thick_{\D(S)}(M\otimes_R S).$$ In particular, for each $N$ in $\Thick_{\D(R)}M$ the natural map $N\to N\otimes_R S $ is an isomorphism in $\D(R)$. 
\end{lemma}
\begin{proof}
The restriction of scalar functor $G:\D(S)\to \D(R)$  is a right adjoint to $-\otimes_R S:\D(R)\to \D(S)$. By assumption, the natural  map $$M\to G(M\otimes_R S)$$ is an isomorphism in $\D(R)$. Hence, (\ref{triangles2}) completes the proof. 
\end{proof}



 \subsection{Support of Cohomology Graded Modules}

Let $\A=\{\A^{i}\}_{i\geq 0}$ be a cohomologically graded, commutative noetherian ring.  Recall that  $\Proj \A$ denotes the set of homogeneous prime ideals of $\A$ not containing $\A^{>0}:=\{\A^i\}_{i>0}.$ For homogeneous elements $a_1,\ldots, a_m\in \A$ define $$\VV(a_1,\ldots, a_m)=\{\mathsf{p}\in \Proj \A: a_i\in \mathsf{p}\text{ for  each }i\}.$$
 For a (cohomologically) graded $\A$-module $\mathsf{X}$, set $$\gsupp{\A} \mathsf{X}:=\{\mathsf{p}\in \Proj \A:\mathsf{X}_\mathsf{p}\neq 0\}.$$ 

 The following properties of (cohomologically) graded $\A$-modules follow easily from the definition of support; see \cite[2.2]{AI}
\begin{proposition}\label{graded}
Let  $\A=\{\A^{i}\}_{i\geq 0}$ be a cohomologically graded, commutative noetherian ring. 
\begin{enumerate}
\item Let $\mathsf{X}$ be a graded $\A$-module and $n\in \Z$. Then $\gsupp{\A}\mathsf{X}=\gsupp{\A}(\shift^n\mathsf{X}).$
\item  Given an exact sequence of graded $\A$-modules $0\to\mathsf{X}'\to \mathsf{X}\to \mathsf{X}''\to 0$ then $$\gsupp{\A} \mathsf{X}=\gsupp{\A} \mathsf{X}'\cup\gsupp{\A} \mathsf{X}''.$$
\item If $X$ is a finitely generated graded $\A$-module, then $\gsupp{\A}\mathsf{X}=\emptyset$ if and only if $\mathsf{X}^{\gg 0}=0.$
\end{enumerate}
\end{proposition}




\section{Cohomology Operators and Support Varieties}\label{kosvar}

\subsection{Fixed Notation}

Throughout this section,   let $Q$ be a commutative noetherian ring. When $Q$ is local, we will let $\n$  denote its maximal ideal and $k$ its residue field.   

Let $I$ be an ideal of $Q$  and fix a generating set $\f=f_1,\ldots, f_n$ for $I$. Set $R:=Q/I$ and   $E:=Q\langle \xi_1,\ldots, \xi_n|\del \xi_i=f_i\rangle$.   The augmentation map $E\to R$ is a map of DG $Q$-algebras. Hence, we consider  DG $R$-modules  as DG $E$-modules via  restriction of scalars along $E\to R$. 

Let $\Poly:=Q[\chi_1,\ldots, \chi_n]$ be a graded polynomial ring where each $\chi_i$ has cohomological degree 2. When $Q$ is local,   set $$\A:=\Poly\otimes_Q k=k[\chi_1,\ldots,\chi_n].$$ 
Define $\Gamma$ to be the graded $Q$-linear dual of $\Poly$, i.e., $\Gamma$ is the graded $Q$-module with  $$\Gamma_i:= \Hom_Q(\Poly^i,Q).$$ Let  $\{y^{(H)}\}_{H\in \N^n}$ be the $Q$-basis  of ${\Gamma}$ dual to  $\{\chi^H:=\chi_1^{h_1}\ldots \chi_n^{h_n}\}_{H\in \N^n}$  the standard $Q$-basis  of $\Poly$. Then  ${\Gamma}$ is a graded  $\Poly$-module via the action $$\chi_i\cdot y^{(H)}:=\left\{\begin{array}{cl} y^{(h_1,\ldots,h_{i-1},h_i-1,h_{i+1},\ldots,h_n)} & h_i\geq 1 \\ 0 & h_i=0 \end{array}\right.$$

\subsection{Cohomology Operators}\label{KCO1}


Let $M$ be a DG $E$-module. A semiprojective resolution  $\e:P\xra{\simeq} M$ over $Q$ such  that  $P$ has the structure of a DG $E$-module and $\e$ is a morphism of DG $E$-modules is called a \emph{Koszul resolution of M}.  A semiprojective resolution of $M$ over $E$ is a  Koszul resolution of $M$, and hence Koszul resolutions exist.  

Let $\e: P\xra{\simeq} M$ be a Koszul resolution of $M$.   Define $U_E(P)$ to be the DG $E$-module with $$U_E(P)^\natural \cong (E\otimes_Q \Gamma\otimes_Q P)^{\natural}$$ and differential given by the formula $$\del=\del^{E}\otimes 1\otimes 1+1\otimes1\otimes \del^{P}+\sum_{i=1}^n (1\otimes \chi_i\otimes \lambda_i-\lambda_i\otimes \chi_i\otimes 1)$$ where $\lambda_i$ denotes left multiplication by $\xi_i.$  
By \cite[2.4]{CD2}, $U_E(P)\to M$ is  a semiprojective resolution over $E$ where the augmentation map is given by 
  $$a\otimes y^{(H)}\otimes x\mapsto \left\{\begin{array}{cl}  a\e(x) & |H|=0 \\ 0 & |H|>1 \end{array}\right.$$
Notice that $U_E (P)$ has a  DG $\Poly$-module structure where $\Poly$ acts on $U_E(P)$ via its action on $\Gamma$. For a DG $E$-module  $N$, $\Hom_{E}(U_E(P),N)$ is  a DG $\Poly$-module and hence, $$\Ext_{E}^*(M,N)=\h(\Hom_{E}(U_E(P),N))$$ is a graded module over $\Poly$. 

\begin{remark}\label{functoriality}
Let $M$ and $M'$ be DG $E$-modules and assume that $\alpha: M\to M'$ is a morphism of DG $E$-modules. 
Let $F$ and $F'$ be semiprojective resolutions of $M$  and $M'$ over $E$, respectively. 
 Since $F$ is semiprojective over $E$, there exists a morphism of DG $E$-modules  $\tilde{\alpha}: F\to F'$  lifting $\alpha$ that is unique up to homotopy. Moreover, $\tilde{\alpha}$ induces a morphism of  DG $E$-modules  $1\otimes 1\otimes \tilde{\alpha}: U_E(F)\to U_E(F')$ that is  $\Poly$-linear and unique up to homotopy. 

In particular, if $F$ and $F'$ are both semiprojective resolutions of a DG $E$-module  $M$, then there exists a DG $E$-module homotopy equivalence $U_E(F)\to U_E(F')$ that is $\Poly$-linear and unique up to homotopy. Thus, the $\Poly$-module structures of $\h(\Hom_E(U_E(F),N))$ and $\h(\Hom_E(U_E(F'),N))$ coincide when $F$ and $F'$ are both semiprojective resolutions of $M$ over $E$. 
\end{remark}

\begin{proposition}\label{funct}
Let $M$ and $N$ be in  $\D(E)$. Then the $\Poly$-module structure on $\Ext_E^*(M,N)$ is independent of choice of Koszul resolution for $M$. Moreover, the $\Poly$-module action on $\Ext_E^*(M,N)$ is functorial in $M$ and  given an exact triangle $M'\to M\to M''\to$ in $\D(E)$, there exists an exact sequence of graded $\Poly$-modules $$\shift^{-1}\Ext_E^*(M',N)\to \Ext_E^*(M'',N)\to \Ext_E^*(M,N)\to \Ext_E^*(M',N).$$
\end{proposition}
\begin{proof}
Let  $P$ be a Koszul resolution of $M$ and $F$ a semiprojective resolution of $M$ over $E$.  There exists a morphism of DG $E$-modules $\tilde{\alpha}: F\to P$ lifting the identity on $M$ which is unique up to homotopy. This induces a DG $E$-module  homotopy equivalence $1\otimes1\otimes \tilde{\alpha}:U_E(F)\to U_E(P)$ that is $\Poly$-linear and unique up to homotopy. Thus, $F$ and $P$ determine the same $\Poly$-module structure on $\Ext_E^*(M,N)$. From Remark \ref{functoriality}, it follows that  the $\Poly$-module structure on $\Ext_E^*(M,N)$ is independent of choice of Koszul resolution for $M$. 

Moreover, by  Remark \ref{functoriality}  the $\Poly$-module structure on $\Ext_E^*(M,N)$ is functorial in $M$. Thus, $\Ext_E^*(-,N)$ sends exact triangles in $\D(E)$ to exact sequences of graded $\Poly$-modules. 
\end{proof}


\begin{chunk}\label{kk}
Assume that $(Q,\n,k)$ is a  local ring and recall that $\A=\Poly\otimes_Q k$. Let $M$ be   in $\D(E)$. The $\Poly$-action on $\Ext_E^*(M,k)$  factors through $\Poly\to \A$, and hence, $\Ext_E^*(M,k)$ is a graded $\A$-module. Therefore, by Proposition \ref{funct}, for any exact triangle $M'\to M\to M''\to$ in $\D(E)$, we get an exact sequence of graded $\A$-modules   $$\shift^{-1}\Ext_E^*(M',k)\to \Ext_E^*(M'',k)\to \Ext_E^*(M,k)\to \Ext_E^*(M',k).$$
\end{chunk}


\begin{lemma}\label{ses}
Assume that $(Q,\n,k)$ is a  local ring and  $M$ is in $\D(E)$. For any  $x\in \n$,  there exists an exact sequence of graded $\A$-modules $$0\to\shift^{-1}\Ext_E^*(M,k)\to \Ext_E^*(M\otimes_Q \emph{Kos}^Q(x),k)\to \Ext_E^*(M,k)\to 0.$$ 
\end{lemma}
\begin{proof}
By  (\ref{kk}),  applying $\Ext_E^*(-,k)$ to the exact triangle $$M\to M\to M\otimes_Q\text{Kos}^Q(x)\to$$ in $\D(E)$ gives us an exact sequences of graded $\A$-modules $$\shift^{-1}\Ext_E^*(M,k)\to \Ext_E^*(M\otimes_Q \text{Kos}^Q(x),k)\to \Ext_E^*(M,k)\xra{x\cdot} \Ext_E^*(M,k).$$ Since $x$ is in $\n$, we obtain the desired result. 
\end{proof}

\begin{proposition}\label{fg}
 Assume that $(Q,\n,k)$ is a regular local ring. For each  $M$  in  $\D^f(E)$,  $\Ext_E^*(M,k)$ is a finitely generated graded $\A$-module.
\end{proposition}
\begin{proof} 


 
 As $\h(M)$ is finitely generated over $Q$ and  $Q$ is regular,  there exists a Koszul resolution $P\xra{\simeq} M$ such that $P$ is a bounded complex of finitely generated free $Q$-modules  (see \cite[2.1]{CD2}). Also, we have an isomorphism of graded $\A$-modules  $$\Hom_E(U_E(P),k)^\natural\cong  \A\otimes_k \Hom_Q(P,k)^\natural.$$ Thus, $\Hom_E(U_E(P),k)$ is a noetherian graded $\A$-module.  As $\A$ is a noetherian graded ring and $\Ext_E^*(M,k)$ is a graded subquotient of $\Hom_E(U_E(P),k)$, it follows that $\Ext_E^*(M,k)$ is  a noetherian graded $\A$-module.  
\end{proof}

\begin{remark}\label{extkk}
Suppose the local ring $(Q,\n,k)$ is regular. By (\ref{kosaction}), $K^Q$ is a DG $E$-module. Assume that    $I\con \n^2$. Left multiplication by $\xi_i$ on $K^Q$ is zero modulo $\n$. Thus, we have an isomorphism of DG $\A$-modules $$\Hom_E(U_E(K^Q),k)\cong \A\otimes_k \Hom_Q(K^Q, k),$$ where  both DG $\A$-modules have trivial differential (see (\ref{kosaction})).  Therefore, there is an isomorphism of graded $\A$-modules  $$\Ext_E^*(k,k)\cong  \A\otimes_k \Hom_Q(K^Q, k).$$In particular, $$\gsupp{\A}\left(\Ext_E^*(k,k)\right)=\Proj\A.$$
\end{remark}


\subsection{Support Varieties}\label{propbksv}

For the rest of the section,  further assume  that $(Q,\n,k)$ is  a regular local ring, $\f$ \emph{minimally} generates $I$,  and $I\con \n$.  Recall that $$\A=\Poly\otimes_Qk=k[\chi_1,\ldots,\chi_n].$$

By Proposition \ref{fg}, $\Ext_E^*(M,k)$ is a finitely generated graded $\A$-module for each  $M$ in  $\D^f(E)$. This leads to the following definition which  recovers  the support varieties of Avramov in \cite{VPD} in the case that $\f$ is a $Q$-regular sequence. The  varieties, defined below, are investigated and further developed in \cite{Pol}. 
\begin{definition}
Let $M$ be in  $\D^f(E)$. Define the \emph{support variety of $M$ over E} to be $$\V_E(M):=\gsupp{\A} \left(\Ext_E^*(M,k)\right).$$
\end{definition}

\begin{theorem}\label{bp}
With the assumptions above, the following hold. 
\begin{enumerate} 
\item Let $M$ and $N$ be in $\D^f(E)$.  If  $N$ is in $\Thick_{\D(E)}M$, then $\V_E(N)\con \V_E(M)$. 
\item For any $M$ in $ \D^f(E)$,  $\V_E(M)=\V_E(M\otimes_Q K^Q).$
\item $\f$ is a regular $Q$-sequence if and only if $\V_E(R)=\emptyset$. 
\end{enumerate}
\end{theorem}
\begin{proof}
Using (\ref{kk}) and Proposition \ref{graded}, it follows that the full subcategory of $\D^f(E)$ consisting of objects $L$ such that $\V_E(L)\con \V_E(M)$ is a thick subcategory of $\D^f(E)$. Therefore, (1) holds. 

Iteratively applying Lemma \ref{ses} and Proposition \ref{graded}(2), establishes (2). 

For (3), first  assume that $\f$ is a $Q$-regular sequence. Hence, the augmentation map $E\to R$ is a quasi-isomorphism. Therefore,  (\ref{augiso}) yields an isomorphism $$\Ext_{E}^*(R,k)\cong \Ext_{R}^*({R},k)=k.$$ Thus, $\V_E(R)=\gsupp{\A}k=\emptyset.$


Conversely,  assume that $\V_E(R)=\emptyset.$ Hence, by Proposition \ref{fg} and  Proposition \ref{graded}(3), \begin{equation}\label{evanish}\Ext_E^{\gg 0}(R,k)=0.\end{equation}
Next, let $\g=g_1,\ldots, g_n$ be a minimal generating set for $I$ such that $\g'=g_1,\ldots, g_c$ is a maximal $Q$-regular sequence in $I$ for some $c\leq g$.  Set   $\ov{Q}:= Q/(\g')$,  $\ov{\g}$ to be the image of $g_{c+1}, \ldots ,g_n$ in $\ov{Q}$, and $\ov{E}:=\Kos^{\ov{Q}}(\ov{\g}).$ Since $\g'$ is a $Q$-regular sequence, we have a quasi-isomorphism of DG $Q$-algebras $E\xra{\simeq} \ov{E}.$ Hence,  (\ref{augiso}) yields  an isomorphism of graded $k$-vector spaces $$\Ext_{\ov{E}}^*(R,k){\cong}\Ext_{{E}}^*(R,k).$$ In particular, $\Ext_{\ov{E}}^{\gg 0}(R,k)=0$ by (\ref{evanish}). Therefore,      $\pd_{\ov{Q}} R<\infty$ (c.f. \cite[B.10]{AINS}). 
Since $R=\ov{Q}/I\ov{Q}$ where $I\ov{Q}$ contains no $\ov{Q}$-regular element, it follows that $I\ov{Q}=0$ (see \cite[1.4.7]{BH} ). Thus, $\g=\g'$, that is, $I$ is generated by a $Q$-regular sequence. Therefore, by \cite[1.6.19]{BH}, $\f$ is  $Q$-regular sequence. 
\end{proof}

\begin{remark}
In \cite{Pol}, a different argument is used to establish Theorem \ref{bp}(c). In fact, the following is shown: \emph{$\f$ is a $Q$-regular sequence if and only if $\V_E(M)=\emptyset$ for some nonzero finitely generated $R$-module M}. \end{remark}

\begin{theorem}\label{l:4} Assume $(Q,\n,k)$ is a regular local ring. Let $R=Q/I$ where $I$ is minimally generated by  $\f=f_1,\ldots, f_n\in\n^2$. Let $E$ be the Koszul complex on $\f$ over $Q$ and set $\A=k[\chi_1,\ldots,\chi_n]$. 
For each homogeneous element  $g\in\A$,  there exists a complex  of $R$-modules $C(g)$  in $\Thick_{\D(R)}k$ such that $$\V_E(C(g))=\VV(g).$$ 
\end{theorem}
\begin{proof} As $Q$ is regular, the Koszul complex $K^Q$ is a free resolution of $k$ over $Q$. Moreover, (\ref{kosaction}) says that $K^Q$ is a  Koszul resolution of $k$. 
By (\ref{KCO1}), there exists a semiprojective resolution $\e: U\xra{\simeq} k$ over $E$  where $U:=U_E(K^Q)$. Let $d$ denote the degree of $g$. Define $$\tilde{C}(g):=\cone(U\xra{g\cdot}\shift^d U).$$ The same proof of  \cite[3.10]{AI} and  Remark \ref{extkk} yields \begin{equation}\label{var}\V_E(\tilde{C}(g))=\VV(g).\end{equation}

Fix  a projective resolution $\delta: P\xra{\simeq} k$   over $R$. Since $U$ is a semiprojective DG $E$-module there exists a morphism of DG $E$-modules $\alpha: U\to P$ such that $\delta\alpha=\e.$  Note that $\alpha$ is a quasi-isomorphism. 

 By Theorem \ref{surj} and (\ref{ext2}), there exists a morphism of complexes of $R$-modules $\gamma: P\to \shift^dk$ such that 
  \begin{equation}\label{comm}\begin{tikzcd}
U \arrow{d}{\simeq}[swap]{\alpha} \arrow{r}{g\cdot}&  \shift^dU\arrow{d}{\shift^d\e}[swap]{\simeq}\\
P  \arrow{r}[swap]{\gamma} &  \shift^dk 
\end{tikzcd}\end{equation}is a diagram of DG $E$-modules that commutes up to homotopy. 
Define $$C(g):=\cone(\gamma).$$ Since $P\simeq k$ and $\gamma$ is a morphism of complexes of $R$-modules, it follows that $C(g)$ is  in $\Thick_{\D(R)}k$. Also, 
as  $\alpha$ are $\shift^d\e$  quasi-isomorphisms and (\ref{comm}) commutes up to homotopy, we get an isomorphism $$C(g) \simeq\tilde{C}(g)$$ in $\D(E).$  Therefore,  Equation (\ref{var}) yields $$\V_E(C(g))=\V_E(\tilde{C}(g))=\VV(g). \qedhere$$
\end{proof}

\section{Virtually Small Complexes}

Let $R$ be a commutative noetherian ring. A complex of $R$-modules $M$ is \emph{virtually small} if $M\simeq 0$ or  there exists a nontrivial object $P$ in  $\Thick_{\D(R)} M\cap\Thick_{\D(R)}R.$ If in addition $P$ can be chosen with $\Supp_R M=\Supp_R P$, we say $M$ is \emph{proxy small}. These  notions were introduced by  Dwyer, Greenlees, and Iyengar in \cite{DGI2} and \cite{DGI}, where the authors apply methods from commutative algebra to homotopy theory and vice versa. 

\begin{sremark}In \cite{DGI2} and \cite{DGI}, the objects of $\Thick_{\D(R)}R$ are called the \emph{small objects of }$\D(R)$. With this terminology, the nontrivial virtually small objects of $\D(R)$ are the complexes that finitely build a nontrivial small object. \end{sremark}

\begin{schunk}\label{tech}
A nontrivial object $M$ of $\D^f(R)$  is virtually small if and only if there exists a maximal ideal $\m=(\x)$ of $R$ such that $\Kos^R(\x)$ is in $\Thick_{\D(R)}M$.  In particular, if $R$ is local, a nontrivial complex $M$ in $\D^f(R)$ is virtually small if and only if $K^R$ is in $\Thick_{\D(R)}M$. This was observed in  \cite[4.5]{DGI}, and is a consequence of a theorem of M. Hopkins \cite{H} and Neeman \cite{N}.
\end{schunk}

 As a matter of notation,  let $\VS(R)$ to be the full subcategory of $\D^f(R)$ consisting of all virtually small complexes.
In the following lemma, the argument for ``(1) implies (2)"  is abstracted from  the proof of  \cite[9.4]{DGI}. 
 \begin{slemma}\label{comp}Let $R$ be a commutative noetherian ring. The following are equivalent:
\begin{enumerate}
\item  $\Thick_{\D(R)}(R/\m)$ is a  subcategory of   $\VS(R)$ for each maximal ideal $\m$ of $R$.
\item $\D^f(R)=\VS(R)$. 
\item $\VS(R)$ is a thick subcategory of $\D(R)$. 
\end{enumerate}
\end{slemma}
\begin{proof}
$(1)\implies (2)$: Let $M$ be a nontrivial object of $\D^f(R)$. Since $M$ is nontrivial, there exists a maximal ideal $\m$ in $\Supp_R M$.  Let $\x$ generate $\m$ and set $$N:= M\otimes_R \Kos^R(\x).$$  
By (\ref{kossup}),  $\Supp_R N= \{\m\}$ and hence, $N$ is in $\Thick_{\D(R)}(R/\m)$ (see (\ref{thick3})).  By assumption, there exists a nontrivial object $P$  in  $\Thick_{\D(R)}N\cap \Thick_{\D(R)}R$. 
Finally, since $N$ is in $\Thick_{\D(R)}M$,   $\Thick_{\D(R)}N$ is a  subcategory of  $\Thick_{\D(R)} M.$ Thus, $P$ is in $\Thick_{\D(R)} M$. That is, $M$ is virtually small. 

$(2)\implies(3)$: Whenever $R$ is noetherian, $\D^f(R)$ is a thick subcategory of $\D(R)$. 

$(3)\implies(1)$: Let $\m$ be a maximal ideal of $R$ and suppose $\x$ generates $\m$. By  (\ref{thick3}), $\Kos^R(\x)$ is in $\Thick_{\D(R)}(R/\m)$. Thus, $R/\m$ is in $\VS(R)$. Since $\VS(R)$ is a thick subcategory of $\D(R)$, it follows that $\Thick_{\D(R)}(R/\m)$ is contained in $\VS(R)$. 
\end{proof}

 \begin{slemma}\label{vsringmap}
 Let $\vp: R\to S$ be a flat morphism of commutative noetherian rings. Suppose $\m$ is a maximal ideal of $R$ such that $\m S$ is a maximal ideal of $S$ and  the canonical map $R/\m \to S/\m S$ is an isomorphism.  
 Then $\Thick_{\D(R)} (R/\m)$ is a subcategory of $\VS(R)$  if and only if $\Thick_{\D(S)} (S/\m S)$ is a subcategory of $\VS(S).$  
 \end{slemma}
 \begin{proof}
Set $K:=\Kos^R(\x)$ where $\x$ generates $\m$. Let $\x'$ denote the image of $\x$ under $\vp$ and set $K':=\Kos^S(\x').$ Hence,  we have an isomorphism of DG $S$-algebras $K'\cong K\otimes_R S$. 

 Assume $\Thick_{\D(R)} (R/\m)$ is a subcategory of $\VS(R)$.   Let $N$  be a nontrivial object of $\Thick_{\D(S)}(S/\m S)$. By Lemma \ref{fl}, there exists a nontrivial complex $M$ in  $\Thick_{\D(R)}(R/\m)$ such that $M\otimes_R S\simeq N$ in $\D(S)$. By assumption and (\ref{tech}), $K$ is in $\Thick_{\D(R)}M$. Hence, $K\otimes_RS$ is in $\Thick_{\D(S)}(M\otimes_RS).$  Since $K'\cong K\otimes_R S$ and $N\simeq M\otimes_R S$, we conclude that $K'$ is in $\Thick_{\D(S)}N$. Thus,   $N$ is in  $\VS(S).$   
 
 Let $M$ be a nontrivial object of $\Thick_{\D(R)}(R/\m)$.   Thus, $M\otimes_R S$ is a nontrivial object of $\Thick_{\D(S)}(S/\m S)$. By assumption and  (\ref{tech}), $K'$ is in $\Thick_{\D(S)}(M\otimes_R S).$ Therefore, \begin{equation}\label{e:7}K'\in \Thick_{\D(R)} (M\otimes_R S). \end{equation} Since the natural map $R/\m\to S/\m S$ is an isomorphism  and  $K$ and $M$ are in  $\Thick_{\D(R)}(R/\m)$, by   applying Lemma \ref{fl} we obtain the following isomorphisms in $\D(R)$ $$K\xra{\simeq} K\otimes_R S\cong  K' \text{ and  }M\xra{\simeq} M\otimes_RS.$$  These isomorphisms  and (\ref{e:7})  imply that  $K$ is in $\Thick_{\D(R)}M$. That is, $M$ is in $\VS(R)$.  \end{proof}





 \begin{sproposition}\label{complete}Let $R$ be a commutative noetherian ring. 
 \begin{enumerate}
 \item  Then $\D^f(R)=\VS(R)$ if and only if $\D^f(R_\m)=\VS(R_\m)$ for every maximal ideal $\m$ of $R$. 
\item  In addition,  assume $(R,\m,k)$ is  local   and let $\widehat{R}$ denote its $\m$-adic completion. Then $\D^f(R)=\VS(R)$ if and only if $\D^f(\widehat{R})=\VS(\widehat{R}).$ 
\end{enumerate}
 \end{sproposition}
 \begin{proof}
 

 
By Lemma \ref{comp}, $\D^f(R)=\VS(R)$ if and only if $\Thick_{\D(R)}(R/\m)$ is a subcategory of $\VS(R)$ for each maximal ideal of $\m$ of $R$.  By Lemma \ref{vsringmap}, the latter holds if and only if  $\Thick_{\D(R_\m)} (\kappa(\m))$ is a subcategory of $\VS(R_\m)$ for each maximal ideal $\m$ of $R$ where $\kappa(\m)=R_\m/\m R_\m.$  Equivalently, $\D^f(R_\m)=\VS(R_\m)$ for each maximal ideal $\m$ of $R$ by Lemma \ref{comp}. Thus, (1) holds.  

Next,  Lemma \ref{vsringmap} yields that $\Thick_{\D(R)} k$ is a subcategory of $\VS(R)$ if and only if $\Thick_{\D(\widehat{R})} k$ is a subcategory of $\VS(\widehat{R})$. Applying Lemma \ref{comp}, finishes the proof of (2).  \end{proof}


\section{The Main Results}


 Let $(R,\m)$ be a commutative noetherian local ring  and let $\widehat{R}$ denote its $\m$-adic completion. The local ring $R$ is said to be a \emph{complete intersection} provided $$\widehat{R}\cong Q/(f_1,\ldots, f_c)$$ where $Q$ is a regular local ring and $f_1,\ldots, f_c$ is a $Q$-regular sequence. In \cite[9.4]{DGI}, the following was established: \emph{if $R$ is a complete intersection every object of $\D^f(R)$ virtually small.  If in addition $R$ is a quotient of a regular local ring,  every object of $\D^f(R)$ is proxy small}. Moreover, the authors posed the following question:

\begin{squest}\label{question}
\cite[9.4]{DGI}  If every object of $\D^f(R)$ is virtually small, is $R$ a complete intersection? 
\end{squest}

 Theorem \ref{mainresult}, below,  answers Question \ref{question} in the affirmative. Much of the work in establishing ``(1) implies (3)" is done in  the proof of a theorem of  Bergh \cite[3.2]{Berg}. The theory of support varieties developed in Section \ref{propbksv} is the key ingredient used to prove ``(2) implies (1)." 
 


\begin{stheorem}\label{mainresult}
Let $R$ be a commutative noetherian local ring. The following are equivalent.
\begin{enumerate}
\item $R$ is a complete intersection.
\item Every  object of $\D^f(R)$ is virtually  small.
 \item Every object of $\D^f(R)$ is proxy small. 
\end{enumerate}
\end{stheorem}
\begin{proof}
$(1)\implies (3)$: Let $M$ be in $\D^f(R)$. In the proof of  \cite[3.2]{Berg}, it is  shown there exist  positive integers $n_1,\ldots, n_t$ and exact triangles in $\D(R)$ 
\begin{align*}
&M\to \shift^{n_1} M\to M(1)\to \\
&M(1)\to \shift^{n_2} M(1)\to M(2)\to \\
 & \hspace{.1in} \vdots  \hspace{.45in}   \vdots  \hspace{.45in}   \vdots  \hspace{.45in}   \vdots\\
&M({t-1})\to \shift^{n_t} M({t-1})\to M(t) \to 
\end{align*} such that $M(t)$ is in $\Thick_{\D(R)}R$. Also, it is clear that $M(t)$ is in  $\Thick_{\D(R)}M$. Since each $n_i\neq 0$,  Lemma \ref{ksp} yields $$\Supp_R M=\Supp_R (M(1))=\ldots =\Supp_R(M(t)).$$ Thus, $M$ is proxy small. 

$(3)\implies (2)$: Clear from the definitions. 

 $(2)\implies (1)$: By Proposition \ref{complete}(2), we may assume that  $R$ is complete. Write $R=Q/I$ where $(Q,\n,k)$ is a regular local ring. Assume  $I$ is minimally generated by $\f=f_1,\ldots,f_n\in \n^2$  and let $E$ be the Koszul complex on $\f$.

Fix $1\leq i \leq n$. By Theorem \ref{l:4},   there exists $C(i)$ in  $\Thick_{\D(R)}k$ with $$\V_E(C(i))=\VV(\chi_i).$$ By assumption,  each $C(i)$ is virtually small. Therefore, (\ref{tech}) implies that    $K^R$ is in  $\Thick_{\D(R)}C(i)$.   Hence,  $$\V_E(K^{R})\con \V_E{(C(i)})=\VV(\chi_i)$$ by Theorem \ref{bp}(1).  Applying Theorem \ref{bp}(2) with $M=R$ yields $$\V_E(R)=\V_E(K^R),$$ and hence, $\V_E(R)\con \VV(\chi_i)$.

Therefore, $$\V_E(R)\con \VV(\chi_1)\cap \ldots \cap \VV(\chi_n).$$ That is, $\V_E({R})=\emptyset$ and so  by Theorem \ref{bp}(3),  $\f$ is a $Q$-regular sequence. Thus, $R$ is a complete intersection. 
\end{proof}

This structural characterization of a complete intersection's derived category yields the following corollary which  was first established by Avramov in \cite{ACI}. 
\begin{scorollary} \label{localize}
Assume a commutative noetherian local ring $R$ is a complete intersection. For any $\p\in \Spec R$,  $R_\p$ is is a complete intersection.
\end{scorollary}
\begin{proof}
For any $\p\in \Spec R$, the functor  $-\otimes_R R_\p: \D^f(R)\to \D^f(R_\p)$ is essentially surjective. Also, the property of proxy smallness localizes.  These  observations and  Theorem \ref{mainresult} complete the proof. 
\end{proof}

Let   $R$ be   a commutative noetherian ring.  We say that $R$ is  \emph{locally a complete intersection} if $R_\p$ is a complete intersection for each $\p\in \Spec R$. By  Corollary \ref{localize}, $R$ is  locally a complete intersection  if and only if $R_\m$ is a complete intersection for every maximal ideal $\m$ of $R$. We obtain the following homotopical characterization of rings that are locally complete intersections. 

\begin{stheorem}
A commutative noetherian ring $R$ is  locally a complete intersection  if and only if every object of $\D^f(R)$ is virtually small.  
\end{stheorem}
\begin{proof} 
As remarked above, $R$ is  locally  a complete intersection  if and only if $R_\m$ is a complete intersection for each maximal ideal $\m$ of $R$. By Theorem \ref{mainresult}, the latter holds if and only if $\D^f(R_\m)=\VS(R_\m)$ for each maximal ideal $\m$ of $R$. Equivalently, $\D^f(R)=\VS(R)$ by Proposition \ref{complete}(1). 
\end{proof}

\end{document}